\documentclass[english]{ourlematema}
\usepackage{amsthm,bm,mathtools, amsmath, amsthm, amssymb, hyperref, color}
\usepackage{MnSymbol} 
\usepackage{graphicx}
\usepackage{caption}
\usepackage[all]{xypic}
\usepackage{verbatim}
\usepackage{chemfig, chemnum}
\usepackage{tikz}
\usepackage[linesnumbered,lined,commentsnumbered]{algorithm2e}
\usepackage{caption}
\usepackage{subcaption}
\usepackage{float}
\usepackage{mathrsfs}
\usepackage{makecell}
\usepackage{array}
\usepackage{fancyvrb}

\newtheorem{theorem}{Theorem}
\numberwithin{theorem}{section}
\newtheorem{proposition}[theorem]{Proposition}

\newtheorem{corollary}[theorem]{Corollary}

\newtheorem{remark}[theorem]{Remark}
\newtheorem{example}[theorem]{Example}
\newtheorem{conjecture}[theorem]{Conjecture}

\theoremstyle{definition}

\newcommand{\PP}{\mathbb{P}}
\newcommand{\RR}{\mathbb{R}}

\newcommand{\CC}{\mathbb{C}}
\newcommand{\ZZ}{\mathbb{Z}}

\title{Kinematic Varieties for Massless Particles}
 
  \author{Smita Rajan}
  \address{%
    University of California at Berkeley \\
    \email{smita\_rajan@berkeley.edu}
}

  \author{Svala Sverrisd\'ottir}
  \address{%
    University of California at Berkeley \\
    \email{svalasverris@berkeley.edu}
}

\author{Bernd Sturmfels}
\address{%
  MPI for Mathematics in the Sciences, Leipzig \\
\email{bernd@mis.mpg.de}
}

 \date{2024/07/04}

\begin{document}
\maketitle
\begin{abstract}
\noindent 
We study algebraic varieties that encode the kinematic data for
$n$ massless particles in $d$-dimensional spacetime
subject to momentum conservation. 
Their coordinates are  spinor brackets,
which we derive from the Clifford algebra 
associated to the Lorentz group. 
This was proposed for $d=5$ in the recent physics literature.
Our kinematic varieties are given by polynomial constraints on tensors
with both symmetric and skew symmetric slices.

\end{abstract}

\section{Introduction}
The real vector space $\RR^d$, when endowed with the Lorentzian
   inner product
$$  x \cdot y \,\, = \,\,- x_1 y_1 + x_2 y_2 + \,\cdots\, + x_n y_n , $$
is known as {\em $d$-dimensional spacetime}.
The {\em Lorentz group} ${\rm SO}(1,d-1)$ consists of all
 $d \times d$ matrices $g$ such that 
${\rm det}\,g\, = \,1$ and $(gx) \cdot (gy) = x \cdot y$ for all
$x,y \in \RR^d$.
The world we live in, with its three space dimensions
and one time dimension, is the case $d=4$.
But, also higher dimensions $d \geq 5$ appear frequently
in physics. 

The present paper was
inspired by  the  article
\cite{PRRVZ}
on amplitudes for $d=5$.
We here address the problem which was
raised in the footnote on page 8 in \cite{PRRVZ},
namely to find the
 {\em non-linear identities between spinor helicity variables}.

We consider a configuration of $n$ particles 
in $d$-dimensional spacetime after complexification. The $i$th particle is represented by its
momentum vector 
$ p_i =(p_{i1},p_{i2},\ldots,p_{id}) \in \CC^d$.
We assume that each particle  is {\em massless}, which means
\begin{equation}
\label{eq:massless} 
p_i \cdot p_i \,\, = \,\,
-p_{i1}^2 \,+\, p_{i2}^2 \,+\, p_{i3}^2 \,+\, \,\cdots\,\, +\, p_{id}^2 \,\,\,=\,\,\, 0 
\quad {\rm for} \quad i=1,2,\ldots,n.
\end{equation}
We also assume that momentum conservation 
$\sum_{i=1}^n p_i = 0$ holds. In coordinates,
\begin{equation}
\label{eq:momentumconservation}
p_{1j} \,+\, p_{2j} \,+\, \,\cdots \,\, + \,p_{nj} \,\,\, = \,\,\, 0
\quad {\rm for} \quad j=1,2,\ldots,d.
\end{equation}
Thus, our parameter space consists of all
solutions to the $n+d$ equations in (\ref{eq:massless}) and
(\ref{eq:momentumconservation}). The pairwise inner products
$s_{ij} = p_i \cdot p_j$ are known as {\em Mandelstam invariants}.
They are invariant under the action of ${\rm SO}(1,d-1)$
on $(p_1,p_2,\ldots,p_n)$.

In Section \ref{sec2} we view the parameter space
through the lens of commutative algebra. We prove
that the ideal generated by (\ref{eq:massless}) and
(\ref{eq:momentumconservation}), denoted $I_{d,n}$, is a~prime  complete intersection, and we  
mention a Gr\"obner basis. The {\em Mandelstam variety}
consists of symmetric $n \times n$ matrices $(s_{ij})$
of rank $\leq d$, with zeros on the diagonal, and 
rows and columns summing to zero.
These constraints also define a prime ideal.
This result generalizes recent findings
for $d=4$ in \cite[Section~4]{EPS}.

In Section \ref{sec3} we turn to the
Clifford algebra ${\rm Cl}(1,d-1)$, and we review how
this gives rise to the spinor
representation of the Lie algebra $\,\mathfrak{so}(1,d-1)$.
We construct the momentum space Dirac matrix $P$
and the charge conjugation matrix $C$.
These matrices have format $2^k \times 2^k$,
for $k = \lfloor d/2 \rfloor$, as seen in Examples 
\ref{ex:Pmatrix} and \ref{ex:Cmatrix}.
Their symmetry properties are derived in
(\ref{eq:Cproperties}) and
Proposition \ref{prop:Cproperties}.

In Section \ref{sec4} we introduce the
spinor brackets $\langle i j \rangle$
and $\langle i j \,k \rangle$. These
are polynomials in the momentum coordinates
$p_{ij}$ and auxiliary parameters $z_{ij}$;
see Example \ref{ex:d3}. These quantities are
understood modulo the ideal $I_{d,n}$ in Section~\ref{sec2}.

Section \ref{sec5} is devoted to
kinematic varieties given by order two brackets $\langle i j \rangle$.
The Grassmannian ${\rm Gr}(2,n)$ and its 
first secant variety arise for $d \leq 5$. For $d \geq 6$, we obtain subvarieties of
determinantal varieties. For $d$ even, the structure of~$P$
leads to the spinor helicity variety of \cite{EPS},
with two types of brackets
$\langle i j \rangle$ and~$[ i j ]$.

In Section \ref{sec6} we study the varieties  of
$n \times (n + 1) \times n$ tensors with entries $\langle i j \rangle$
and $\langle i j \,k \rangle$. These tensors have the symmetry properties in
Theorem \ref{thm:basicspinprops}.
We present some theoretical results, lots of computations, and
various conjectures.

\section{Massless Particles with Momentum Conservation}
\label{sec2}

We fix the polynomial ring $ \CC[p]$ in the
$nd$ variables $p_{ij}$, and we write
$I_{d,n} \subset \CC[p]$ for the ideal generated by 
the $n$ quadrics in (\ref{eq:massless}) and the $d$ linear forms in
(\ref{eq:momentumconservation}).
The variety $V(I_{d,n})$ is the parameter space
for $n$ massless particles with momentum conservation.
We show that this variety is irreducible of the expected dimension.

\begin{theorem} \label{thm:parameters}
$I_{d,n}$ is prime and is a complete intersection, if ${\rm max}(n,d) \ge 4$.
\end{theorem}

\begin{proof}
We first give the proof for $d \geq 5$. For $d=4$
see \cite[Theorem 4.5]{EPS}. We eliminate $p_n$ by solving the linear equations for
$p_{n1},\ldots,p_{nd}$. This leaves the quadric 
$f = (\sum_{i=1}^{n-1} p_i) \cdot (\sum_{i=1}^{n-1} p_i)$ and the ideal
$J = \langle p_1 \cdot p_1 ,\ldots, p_{n-1}\cdot p_{n-1}\rangle$
in the polynomial ring $\CC[p_1,\ldots,p_{n-1}]$.
Our aim is to show that  $\langle f \rangle$ is a prime ideal in
$$ R \,\,=\,\, \CC[p_1,\ldots,p_{n-1}]/J \,\, \simeq \,\,
\frac{\CC[p_1]}{\langle p_1\cdot p_1\rangle }\,\otimes_\CC\,  \frac{\CC[p_2]}{\langle p_2\cdot p_2 \rangle}
\,\otimes_\CC\, \,\cdots\, \, \otimes_\CC \,
\frac{\CC[p_{n-1}]}{\langle p_{n-1}\cdot p_{n-1} \rangle}.
$$
By \cite[Exercise II.6.5]{Har}, 
$ {\rm Spec}( \CC[p_i]/\langle p_i \cdot p_i \rangle ) $
has trivial divisor class group  for $d \geq 5$, and  is normal for
 $ d \geq 3$. Then, \cite[Proposition II.6.2]{Har} implies  that
$\CC[p_i]/\langle p_i \cdot p_i \rangle $
 is a UFD for $d \geq 5$.  A tensor product of UFDs over $\CC$ is a UFD. Hence
 $R$ is a UFD. It therefore suffices to show that $f$ is irreducible in $R$.
This holds because  $f$ is equivalent to 
$x_1^2 + \cdots + x_{n-1}^2$ by a linear change of coordinates.

The UFD $R$ has Krull dimension $(n-1)(d-1)$.
Its principal ideal $\langle f \rangle$ has height one.
So, the ring $\CC[p]/I_{d,n} \simeq R/ \langle f \rangle$ has Krull dimension
$(n-1)(d-1)-1 = nd-(n+d)$. 
We conclude that the ideal $I_{d,n}$ is a complete intersection.

We now give the proof for $d = 3$ and $n \ge 4$. 
 We claim that $S = \mathbb{C}[p] /I_{3,n}$ is a normal ring, i.e.~the localization $S_\mathfrak{p}$ is an integrally closed domain
 for each prime ideal $\mathfrak{p} \subset S$. To do so we apply Serre's criterion for normality in two steps.
    
    \begin{enumerate}
        \item  The ring $R = \mathbb{C}[p_1,\dots,p_{n-1}]/J$ is normal and has dimension $2n-2$.       
          The associated primes of   the principal ideal $\langle f \rangle$ all have height $1$. Therefore, $S \simeq R/ \langle f \rangle$ has dimension $2n - 3$. So, $I_{3,n}$ is a complete intersection ideal. Hence $S$ is Cohen-Macaulay and satisfies Serre's depth condition, i.e.~$S_\mathfrak{q}$ has depth at least $2$ for every minimal prime ideal $\mathfrak{q} \subset S$ of height at least~$2$.  
        \item We claim that  $S = \mathbb{C}[p] /(J+\langle f \rangle)$ is regular in codimension $1$. To prove this, it suffices to show that the singular locus of $S$ has codimension at least $2$. 
      Up to scaling the columns,   the Jacobian matrix of $ J+\langle f \rangle$ equals

\vspace{-0.1in}
             \begin{smaller}
  $$      \begin{bmatrix}
            p_{11} & p_{12} & p_{13} & 0 & 0 & 0 &\cdots & 0 & \cdots & 0\\ 
                       0 & 0 & 0 & p_{21} & p_{22}   & p_{23} & \cdots & 0 & \cdots & 0\\               
                  \vdots   &  \vdots & \vdots & \vdots   & \vdots & \vdots & \ddots & \vdots & \ddots & \vdots\\
                      0    & 0 &   0    &    0   & 0   &   0    &  \cdots & p_{n-1,1}& p_{n-1,2} & p_{n-1,3}\\
                    \sum\limits_{\ell} p_{\ell 1}    & \sum\limits_{\ell} p_{\ell 2} &   \sum\limits_{\ell} p_{\ell 3}      &  
                              \sum\limits_{\ell} p_{\ell 1}    & \sum\limits_{\ell} p_{\ell 2} &   \sum\limits_{\ell} p_{\ell 3}      & 
                              \cdots &
                                        \sum\limits_{\ell} p_{\ell 1}    & \sum\limits_{\ell} p_{\ell 2} &   \sum\limits_{\ell} p_{\ell 3}      \\
        \end{bmatrix}. $$
        \end{smaller}
        
        The singular locus of $S$ is cut out by the maximal minors of
        this matrix. Pick $i_1, i_2, \dots, i_{n-1}, j_k \in \{1,2,3\}$  with $j_k \neq i_k$.
         Then we get the $n \times n$ minor
 $$
       \left(\prod_{t = 1, t \neq k}^{n-1} p_{t i_t} \right)\left( p_{k i_k} 
       \sum_{\ell=1}^{n-1} p_{\ell j_k } - p_{k j_k}  \sum_{\ell=1}^{n-1} p_{\ell i_k}  \right).
$$
Since $n \geq 4$, 
there are enough such products to cut out a locus of codimension $\geq 2$ in $\operatorname{Spec}(S)$.
           This fails for $n=3$. So, the ring $S$ satisfies Serre's regularity condition, 
           that is, $S_\mathfrak{p}$ is a DVR for any prime $\mathfrak{p} \subset S$ of height $\leq 1$.
    \end{enumerate}

    Serre's depth and regularity conditions imply that $S$ is a normal ring. Noetherianity implies
     $S$ is a finite product of normal domains: $S \simeq S_1 \times \dots \times S_r$. Since $S$ is standard graded ring, it has no non-trivial idempotents, and therefore $r=1$. 
        In conclusion, we have shown that $S$ is a domain, and $I_{3,n}$ is a prime ideal.
\end{proof}

\begin{example}
The hypothesis ${\rm max}(n,d) \ge 4$ 
is needed in Theorem \ref{thm:parameters}.
The following session in {\tt Macaulay2} \cite{M2} shows that
$I_{3,3}$ is primary but not prime:
\begin{verbatim}

i1 : R = QQ[p11,p12,p13,p21,p22,p23,p31,p32,p33];
i2 : I = ideal(p11+p21+p31, p12+p22+p32, p13+p23+p33,
      p11^2-p12^2-p13^2,p21^2-p22^2-p23^2,p31^2-p32^2-p33^2);
i3 : codim I, degree I
o3 = (6, 8)
i4 : isPrime I, isPrimary I
o4 = (false, true)
\end{verbatim}
\end{example}

\begin{remark} \label{rmk:GBsimple}
Fix the reverse lexicographic order on $\CC[p]$ where
the entries of the $n \times d$ matrix $(p_{ij})$ are sorted row-wise.
We find that the
reduced  Gr\"obner basis of $I_{d,n}$ 
stabilizes for $n,d \geq 3$. Namely,
the initial monomial ideal equals
$$ \begin{matrix} {\rm in}(I_{d,n}) & = &
\langle \,p_{1j} : j=1,\ldots,d \,\rangle \,\,+ \,\,
\langle \,p_{i1}^2 : i = 2,\ldots,n \,\rangle \qquad \\
 & &  \, + \, \,\langle\,
p_{21}p_{31}, \,
p_{22} p_{31} p_{32}, \,
p_{22}^2 p_{31},\, p_{23}^2p_{32}^2,\,
 p_{23}^2 p_{31} p_{32} \,\rangle.
 \end{matrix}
$$
Remarkably, the last five monomial generators
are independent of $d$ and $n$.
\end{remark}

Remark \ref{rmk:GBsimple} suggests that it might be
easy to parametrize the variety $V(I_n)$. We found that
this is not the case at all. A naive idea is to express
the variables $p_{1j}$ and $ p_{i1}$ in terms of the
entries of the $(n-1) \times (d-1)$ matrix $p' = (p_{ij})_{i,j \geq 2}$.

\begin{remark} \label{rmk:Elimhard}
The elimination ideal $I_{d,n} \cap \CC[p']$ is principal.
Its generator is a large polynomial of degree $2^{n-1}$.
For instance, for $n=4,d=5$, this octic has $4671$ terms.
This hypersurface is a notable obstruction to any
naive parametrization.

For $d=4$, the momentum space twistors introduced by Hodges in \cite{Hod}
yield a beautiful parametrization. It would be very interesting to
extend this to $d \geq 5$.
\end{remark}

Physical properties of our $n$ particles are
expressions in $p_1,\ldots,p_n$ that are
invariant under the action of the orthogonal group
$G = {\rm O}(1,d-1)$. The ring of $G$-invariants in $\CC[p]$
is generated by the {\em Mandelstam invariants} 
 $s_{ij} = p_i \cdot p_j$. 
 The ideal $I_{d,n}$ is fixed under this action,
 and we are interested in the invariant ring
 \begin{equation}
 \label{eq:Ginvariants}  (\CC[p]/I_{d,n})^G  \,\, = \,\, \CC[S]/M_{d,n} . 
 \end{equation}
Here $S = (s_{ij})$ is a symmetric $n \times n$ matrix. Its
entries are the variables in the polynomial ring $\CC[S]$. We shall
characterize the prime ideal $M_{d,n}$. The~variety
\begin{equation}
\label{eq:mandelstamvariety}
 V(M_{d,n}) \,\,=\,\, {\rm Spec}\bigl((\CC[p]/I_{d,n})^G\bigr) \,\,=\,\, V(I_{d,n}) // G  
 \end{equation}
is the {\em GIT quotient}, whose points are the $G$-orbits
of configurations $(p_1,\ldots,p_n)$.
We call (\ref{eq:mandelstamvariety}) the {\em Mandelstam variety}.
The case $d=4$ was studied in \cite[Section~4]{EPS}.

\begin{theorem}
Let $n \geq 2$ and $d \geq 4$.  The Mandelstam ideal $M_{d,n}$ is equal to
\begin{equation}
\label{eq:mandelideal}  \!\!   \langle s_{11}, s_{22}, \ldots, s_{nn} \rangle \,+\, \bigl\langle
(d\!+\!1) \times (d\!+\!1) \,\,\hbox{minors of}\,\, S  \,\bigr\rangle \,+\,
\bigl\langle \sum_{j=1}^n s_{ij} \,\,:\,\, i = 1,\ldots ,n \bigr\rangle.
\end{equation}
In particular, this ideal is prime.
The dimension of the Mandelstam variety~equals
\begin{equation}
\label{eq:mandeldim}
 {\rm dim}(V(M_{d,n})) \,\, = \,\, 
 nd-n-d \,-\, \binom{d}{2} \,\,=\,\, {\rm dim}(V(I_{d,n})) \,-\, {\rm dim}({\rm O}(1,d-1)). 
 \end{equation}
\end{theorem}

\begin{proof}
The dimension formula (\ref{eq:mandeldim}) 
holds because the orthogonal group  acts
faithfully on the complete intersection $V(I_{d,n})$.
Consider the polynomials that generate the
three  ideals in (\ref{eq:mandelideal}). These polynomials
vanish for massless particles in $d$ dimensions subject to
momentum conservation. It hence suffices to show
that the sum in (\ref{eq:mandelideal}) is a prime ideal and that it
has the correct dimension.
This proof rests on the results for $I_{d,n}$ in Theorem \ref{thm:parameters}.
The details are presented for $d=4$ in the proof of 
\cite[Theorem 4.5]{EPS}. The general case $d \geq 5$ is analogous.
\end{proof}

\begin{remark} \label{rmk:maxideal}
If $n \leq 3$ then (\ref{eq:mandelideal}) is the maximal ideal, i.e.~$\,M_{d,n} \,=\,\langle s_{ij}\, : \, 1 \leq i < j \leq n \rangle$.
\end{remark}

\section{Clifford Algebras and Spinors}
\label{sec3}

Our aim is to represent the kinematic data
for  $n$ particles  in terms of spinors.
This encoding rests on the
Clifford algebra ${\rm Cl}(1,d-1)$,
which is associated with the matrix
$\eta=\mathrm{diag}(-1,1,\dots,1)$. 
This section offers an introduction to
Clifford algebras and the spin representation of ${\rm SO}(1,d-1)$.
For a systematic account see Chevalley's book \cite{Che}.
Our exposition is inspired by the article~\cite{RdT}.

The {\em Clifford algebra} ${\rm Cl}(1,d-1)$ is the
free associative algebra $\CC \langle \gamma_1,\ldots,\gamma_d \rangle$
modulo the two-sided ideal generated by 
$\,\gamma_i \gamma_j + \gamma_j \gamma_i - 2\eta_{ij}$ for $1 \leq i,j \leq d$.
Its associated graded algebra  is the exterior algebra $\wedge^* \CC^d$, so
${\rm dim}_\CC\, {\rm Cl}(1,d-1)  =  2^d$. A basis is given by
the square-free words $\gamma_{i_1} \gamma_{i_2} \cdots \gamma_{i_k}$
for $1 \leq i_1 < i_2 < \cdots < i_k \leq d$.

For applications in physics, one uses the representation of the Clifford algebra
${\rm Cl}(1,d-1)$ by {\em Dirac matrices} \cite[Section 4]{RdT}.
These matrices have size $2^k \times 2^k$, where
 $d=2k$ if $d$ is even and $d=2k+1$ if $d$ is odd. 
 We shall define the Dirac matrices 
  $\Gamma_1, \Gamma_2, \ldots, \Gamma_d$   recursively.
Starting with $d=2$ and $k=1$, we set
\begin{equation}
\label{eq:Omega1}     \Gamma_1 \, = \,
\begin{small}
    \begin{bmatrix}
        0 & 1 \\
        -1 & 0 \\
    \end{bmatrix} \end{small} \quad {\rm and} \quad
    \Gamma_2 \, = \, \begin{small}
    \begin{bmatrix}
        0 & 1 \\
        1 & 0
    \end{bmatrix}. \end{small}
\end{equation}

For $d =2k \geq 4$, we take tensor products
of smaller $\Gamma$ matrices with  Pauli matrices.
Namely, if $\Gamma_{k-1,i}$ is the $i$th Dirac matrix for $d=2k-2$
then we define
\begin{equation}
\label{eq:Omega2}
  \Gamma_{i} \,\,=\,\, \Gamma_{k-1,i} \,\otimes \,
\begin{small}        \begin{bmatrix}
            -1 \!\! &\!\!\! 0 \\
            0 \!\! &\!\!\! 1
        \end{bmatrix}
\end{small}        \,
        \,\,\,{\rm for}\,\,\, 1\leq i \leq d-2,\,
\end{equation}
\begin{equation}
\label{eq:Omega3}
                 \Gamma_{d-1} \,\,=\,\, {\rm Id}_{2^{k-1}} \,\otimes \,
      \begin{small}  \begin{bmatrix}
            0 \!\! &\!\! 1 \\
            1 \!\! &\!\! 0
        \end{bmatrix} \end{small} \,, \,\,\,
        \Gamma_{d} \,\,= \,\,{\rm Id}_{2^{k-1}} \,\otimes \,
\begin{small}        \begin{bmatrix}
            0 \!\! & \!\! \!\! -i \\
            i \!\! &\!\! 0
        \end{bmatrix}. \end{small}
\end{equation}
When $d = 2k+1$ is odd, we  construct the first $d - 1$ Dirac matrices as above,
 and we then add one additional Dirac matrix as follows:
\begin{equation}
\label{eq:Omega4}
    \Gamma_{d}\,\,=\,\,-i^{k-1} \cdot \Gamma_1 \Gamma_2 \,\cdots\, \Gamma_{d - 1}.
\end{equation}
One checks by induction on $k$ that this yields
a representation of ${\rm Cl}(1,d-1)$.

 \begin{proposition} \label{prop:clifford}
 The Dirac matrices satisfy the Clifford algebra relations, i.e.~we have
$\,{\Gamma_{\! 1}}^{\!\! 2} \,=\, -\,{\rm Id}_{2^k} $, 
$\,{\Gamma_{\! j}}^{\! 2} \,=\, \,{\rm Id}_{2^k}\,$ for $j \geq 2$, and
$\,\Gamma_i \Gamma_j + \Gamma_j \Gamma_i \,=\, {\rm 0}_{2^k}$ for $i \not= j$.
\end{proposition}

In what follows we work with a variant of the matrices
above where the rows and columns have been permuted
to achieve a desirable block structure.
Namely, for $d = 2k$ even, each Dirac matrix
$\Gamma_i$ is anti-block diagonal of size $2^k \times 2^k$ with two blocks each of size 
$2^{k - 1} \times 2^{k-1}$. This corresponds to the fact that the representation 
of ${\rm Cl}(1,d-1)$ is reducible when $d$ is even. In this case, it splits into two Weyl representations, 
namely the left handed and right handed spinors \cite[Section 4.2.1]{RdT}. For $d = 2k + 1$ odd, the representation is irreducible, and we augment
our basis by the diagonal matrix
 $\Gamma_{d} = {\rm diag}(1, \dots, 1, - 1, \dots, -1)$. 

Consider a particle with momentum vector $p\in \CC^d$. We define its
   \textit{momentum space Dirac matrix}  to be the following linear combination of the Dirac matrices:
\begin{equation}
\label{eq:Pmatrix}
P \,\,=\,\, -p_1 \Gamma_1 + p_2 \Gamma_2  + p_3 \Gamma_3 + \cdots + p_{d}\Gamma_{d}.
\end{equation}
We next write $P$ explicitly. This illustrates the block structure mentioned above.

\begin{example}[$k = 1,2,3$]  \label{ex:Pmatrix}
    We take a look at the momentum space Dirac matrices $P = P^{(d)}$ for
    $d \leq 7$.
           When $d=2k$ is even, we get the anti-block diagonal matrices 
    $$
         P^{(2)} = \begin{bmatrix}
        0 & -p_1+p_2\\
        p_1+p_2 & 0
    \end{bmatrix}, \quad  P^{(4)} = 
    \begin{small} \begin{bmatrix}
        0 & 0 & p_1-p_2 & p_3-ip_4 \\
        0 & 0 & p_3+ip_4 & p_1+p_2\\
        -p_1-p_2 &  p_3-ip_4 & 0 & 0\\
        p_3+ip_4 & -p_1+p_2 & 0 & 0
    \end{bmatrix}, \end{small}
    $$
        $$
    P^{(6)} = \begin{footnotesize} \begin{bmatrix}
        0 & 0 & 0 & 0 &\! \!\!-p_1+p_2 & 0 & \!\!\! -p_3+ip_4 &\!\! p_5-ip_6\\
        0 & 0 & 0 & 0 & 0 & \!\! -p_1+p_2 & p_5+ip_6 & \!\!p_3+ip_4\\
        0 & 0 & 0 & 0 &\!\!\!\!\! -p_3-ip_4 & p_5-ip_6 &\!\!\! -p_1-p_2 & 0\\
        0 & 0 & 0 & 0 & p_5+ip_6 & p_3-ip_4 & 0 & \!\!\! \!\! -p_1-p_2\\
        p_1+p_2 & 0 & \!\!\!\!\! -p_3+ip_4 & p_5-ip_6 & 0 & 0 & 0 & 0\\
        0 & p_1+p_2 &\! p_5+ip_6 & p_3+ip_4 & 0 & 0 & 0 & 0\\ 
        -p_3-ip_4 & \!\!\! p_5-ip_6 & p_1-p_2 & 0 & 0 & 0 & 0 & 0\\ 
        p_5+ip_6 &\!\!\! p_3-ip_4 & 0 & p_1-p_2 & 0 & 0 & 0 & 0
    \end{bmatrix}. \end{footnotesize}
    $$
The matrix $P^{(d)}$ for $d = 2k + 1$ odd is obtained from the matrix $P^{(d - 1)}$ by adding 
the diagonal matrix $\,p_{d}\cdot {\rm diag}(1, \dots, 1, - 1, \dots, -1)$. For example, we have
    $$
    P^{(5)} \,=\, \begin{small} \begin{bmatrix}
        p_5 & 0 & p_1-p_2 & p_3-ip_4 \\
        0 & p_5 & p_3+ip_4 & p_1+p_2\\
        -p_1-p_2 &  p_3-ip_4 & -p_5 & 0\\
        p_3+ip_4 & -p_1+p_2 & 0 & -p_5
    \end{bmatrix}. \end{small}
    $$
\end{example}

We now return to arbitrary spacetime dimension $d \geq 2$.
It follows from Proposition \ref{prop:clifford} that the
 momentum space Dirac matrix $P$ satisfies the identity
\begin{equation}
\label{eq:P^2}
P^2 \,\,=\,\, (-p_1^2 + p_2^2 + \cdots + p_{d}^2)\,{\rm Id}_{2^k}.
\end{equation}
This implies the formula
$\,\det \,P\,=\, (p_1^2 - p_2^2 - \cdots - p_{d}^2)^{2^{k - 1}}$
for the determinant of~$P$.

\begin{corollary} \label{cor:rankP}
For massless particles in $d$ dimensions, the momentum space Dirac matrix  $P$ 
squares to $0$
and its rank equals half of its size, 
i.e.\ ${\rm rank} \,\,P = 2^{k - 1}$. 
\end{corollary}

The Dirac representation of the Clifford algebra ${\rm Cl}(1,d-1)$ gives rise
to the spin representation of the Lie algebra $\mathfrak{so}(1,d-1)$.
We consider the commutators
\begin{align}\label{lie algebra rep}
    \Sigma_{jk}\,\,=\,\,\frac{1}{4}[\Gamma_j,\Gamma_k].
\end{align}
These matrices $\Sigma_{jk}$ define a representation of $\mathfrak{so}(1,d-1)$ 
because they satisfy
\begin{align}
    [\Sigma_{ij},\Sigma_{kl}]\,\,=\,\,\eta_{jk}\Sigma_{il}+\eta_{il}\Sigma_{jk}-\eta_{jl}\Sigma_{ik}-\eta_{ik}\Sigma_{jl}.
\end{align}
Indeed, these are the commutation relations satisfied by the
matrices in the
standard basis of the Lie algebra $\mathfrak{so}(1,d-1)$.
One obtains the spin representation of the Lie group
${\rm SO}(1,d-1)$ on $\CC^{2^k}$ by taking
the matrix exponentials ${\rm exp}(\Sigma_{jk})$.

Another important player in our story is the \textit{charge conjugation matrix} $C$. 
This is an invertible $2^k \times 2^k$ matrix with entries in $\ZZ[i]$.
It represents an equivariant  linear map from the spinor representation of $\mathfrak{so}(1,d - 1)$
 to its dual representation.  
 The characteristic properties of the charge conjugation matrix $C$~are:
\begin{equation}
\label{eq:Cproperties}
CP = -P^TC \,\,\, \text{ if $d=2k$ is even}, \quad CP = (-1)^kP^TC  \,\,\, \text{ if $d=2k+1$ is odd}.
\end{equation}
We explicitly realize $C = C^{(d)}$ as a matrix
 whose nonzero entries are $1,-1,i,-i$.
 
\begin{example}[$k = 1,2,3$] \label{ex:Cmatrix}
    For $d = 2,3$ we take the same  skew symmetric matrix:
        $$
    C^{(2)} \,\,= \,\,C^{(3)}\,\, = \,\,\begin{bmatrix}
        0 & 1\\
        -1 & 0
    \end{bmatrix}.
    $$
    For $d = 4,5$ we obtain skew symmetric matrices with a block diagonal structure:
    $$
    C^{(4)} \,\,= \,\,\begin{small} \begin{bmatrix}
        0 & -i & 0 & 0\\
        i & 0 & 0 & 0\\
        0 & 0 & 0 & -i\\
        0 & 0 & i & 0
    \end{bmatrix}, \end{small} \quad C^{(5)} \,\,=\,\,
       \begin{small} \begin{bmatrix}
        0 & -i & 0 & 0\\
        i & 0 & 0 & 0\\
        0 & 0 & 0 & i\\
        0 & 0 & -i & 0
    \end{bmatrix}. \end{small}
    $$
    Finally, for $d = 6,7$ we have the same charge conjugation matrix:
      $$
    C^{(6)} \,\,= \,\,C^{(7)} \,\, =\,\,
    \begin{footnotesize} \begin{bmatrix}
        0 & 0 & 0 & 0 & 0 & 1 & 0 & 0\\
        0 & 0 & 0 & 0 & -1 & 0 & 0 & 0\\
        0 & 0 & 0 & 0 & 0 & 0 & 0 & -1\\
        0 & 0 & 0 & 0 & 0 & 0 & 1 & 0\\
        0 & -1 & 0 & 0 & 0 & 0 & 0 & 0\\
        1 & 0 & 0 & 0 & 0 & 0 & 0 & 0\\
        0 & 0 & 0 & 1 & 0 & 0 & 0 & 0\\
        0 & 0 & -1 & 0 & 0 & 0 & 0 & 0
    \end{bmatrix}. \end{footnotesize}
    $$ 
     This is a symmetric $8 \times 8$ matrix with an 
    anti-block diagonal structure.
\end{example}

The following result explains the matrix structures we saw
in Example~\ref{ex:Cmatrix}.

\begin{proposition} \label{prop:Cproperties}
    The following properties hold for the charge conjugation matrix $C$ associated with particles
    in spacetimes of dimension $d=2k$ and $d=2k+1$:
    \begin{enumerate}
        \item  $C$ is symmetric when $k \equiv 0,3 \!\! \mod{4}$; otherwise it is skew symmetric.
        \item $C$ is block diagonal when $k \equiv 0 \!\!\! \mod{2}$; otherwise it is anti-block diagonal.
        \item The $2^{k - 1} \times 2^{k - 1}$ blocks of $C$ are skew symmetric when $k \equiv 2,3 \!\! \mod{4}$; otherwise the blocks are symmetric.
    \end{enumerate}  
\end{proposition}

\begin{proof}
    The charge conjugation matrix can be written as follows:
\begin{equation}
\label{eq:CCM}
\begin{matrix}
    C & = & \Gamma_{d + 1} \Gamma_4 \Gamma_6 \,\cdots\,
    \Gamma_{d-2} \Gamma_{d} \Gamma_1 \,\,\,\,\, \text{ if $d \equiv 0\! \!\!\!\mod{4}$,}
    \\
    C & = & \!\!\!\!\!\! \Gamma_4  \Gamma_6 \,\cdots\, \Gamma_{2k-2} \Gamma_{2k} \Gamma_1 \,
    \quad \text{ otherwise.}
    \end{matrix}
\end{equation}
    Here, we set $\Gamma_{d + 1} = -i^{k - 1} \Gamma_1 \cdots \Gamma_d$,
    which is analogous to (\ref{eq:Omega4}).    
    One can check that this $C$ defines an equivariant map to the dual representation and 
    (\ref{eq:Cproperties}) holds.
    
The Dirac matrices are either symmetric or skew symmetric. Namely, we~have
\begin{equation}
\label{eq:eithersymmetric}
    \Gamma_1^T\,=\, -\Gamma_1, \quad \Gamma_2^T\, =\,\Gamma_2,  \quad {\rm and} \quad
     \Gamma_{2i - 1}^T \,= \, \Gamma_{2i - 1},\,\,\,\,
     \Gamma_{2i}^T = -\Gamma_{2i} \quad  {\rm for} \,\,\, i \,\geq \,2 .
\end{equation}
First suppose 
  $d \not\equiv 0 \!\! \mod{4}$. Then (\ref{eq:eithersymmetric}) and the Clifford algebra relations imply
        $$
    C^T \,\,=\,\, (-1)^k \Gamma_1 \Gamma_{2k} \Gamma_{2k-2}\, \cdots\,\Gamma_6 \Gamma_4 
    \,\,=\,\, (-1)^k(-1)^{\frac{k(k - 1)}{2}} C \,\,= (-1)^{\frac{(k + 1)k}{2}} C.
    $$
    We conclude that the matrix  $C$ is symmetric if $k \equiv 3 \!\! \mod{4}$
    and it is skew symmetric if $k \equiv 1,2 \!\! \mod{4}$.
          Now suppose that $d \equiv 0 \!\! \mod{4}$. In this case, we find
    $$
    C^T\,\, =\,\, (-1)^k \Gamma_1 \Gamma_{d} \Gamma_{d-2}\, \cdots\, \Gamma_6 \Gamma_4 \Gamma_{d + 1} \,\,
    = \,\,(-1)^k (-1)^{\frac{(k + 1)k}{2}} C \,\,=\,\, (-1)^{\frac{(k + 3)k}{2}} C.
    $$
    Hence $C$ is symmetric if $k \equiv 0 \!\! \!\mod{4}$; otherwise it is skew symmetric. 
    Bearing in mind that $d \in \{2k,2k+1\}$, this completes
    the proof of part 1 in Proposition~\ref{prop:Cproperties}.
    
    All  Dirac matrices $\Gamma_i$ are anti-block diagonal, except the last one when $d$ is odd. Since $C$ always has $k$ anti-block diagonal terms, we find that $C$ is block diagonal if $k \equiv 0 \!\! \mod{2}$. Otherwise it is 
    anti-block diagonal. This proves part~2.

The verification of  part 3 is similar, but it is a bit more technical.
\end{proof}

\section{Spinor Brackets}
\label{sec4}

We now return to our primary goal, namely to model interactions among
$n$ massless particles in $d$-dimensional spacetime. This is
based on the Dirac matrices in Section \ref{sec3}. Recall that the $i$th 
particle is the vector $p_i = (p_{i1},p_{i2},\ldots,p_{id})$.
As in Section \ref{sec2}, we assume that the tuple 
$(p_1,\ldots,p_n) \in \CC^{nd}$ lies in  the variety~$V(I_{d,n})$.
   
 The {\em momentum space Dirac matrix} for the $i$th particle is defined as
$$ P_i \,\,\,=\,\,\, -p_{i1}\Gamma_1 \,+\,   p_{i2} \Gamma_2\,+\,
p_{i3} \Gamma_3\,+\,\cdots\,+\,p_{id} \Gamma_d.
 $$
 This matrix has format $2^k \times 2^k$, where $k = \lfloor d/2 \rfloor$, 
  and its entries are linear forms in $p_i$.  The rank of $P_i$ equals $2^{k-1}$, 
  since the particle is massless, by Corollary \ref{cor:rankP}.
  The Clifford algebra relations
  imply the following anti-commutator identities:
\begin{equation}
\label{eq:cliff_to_s} 
P_i P_j \,+ \, P_j P_i \,\,\, = \,\,\, 2p_i \cdot p_j \, {\rm Id}_{2^k} \,\,=\,\,
2s_{ij} \,  {\rm Id}_{2^k}.
\end{equation}

 For each $i \in \{1,2,\ldots,n\}$, we now introduce $2^{k-1}$
additional variables $z_{ij}$.
These variables  parameterize the column space of $P_i$,
using the basis consisting of
the first $2^{k - 2}$ and last $2^{k - 2}$ columns of $P_i$.
 We fix the parameter vector
 \begin{equation}
    z_i \,\,=\,\, (\,
        z_{i1} ,\,z_{i2} ,\,\ldots \,, \,z_{i ,2^{k - 2}} \,,\,  0 ,\,0, \,\ldots \,,\, 0\, , \,z_{i ,2^{k - 2} + 1} 
        ,\,        z_{i,2^{k - 2} + 2},\,        \ldots \,, z_{i, 2^{k - 1}} \,)^T.
\end{equation}
Here it is assumed that $k \geq 2$.
In the small special case $k=1$ we set $z_i = (z_{i1},0)^T$.

We use Dirac's ket-notation for a general
vector in the column space of $P_i$:
\begin{equation}\label{eq:paramofcolumnspace}
    |\, i\, \rangle \,\,= \,\,P_i\,z_i.
\end{equation}
The bra-notation $\,\langle \, i \,|\,$
is used for the corresponding row vector $|\, i \, \rangle^T$.
Thus $|\, i\, \rangle \,$ and $\,\langle \, i \,|\, $
are vectors that depend on $d + 2^{k-1}$ parameters.
They represent  particle~$i$.

\begin{remark}
In this paper we restrict ourselves to massless particles.
If particle $i$ is massive then its momentum vector
$p_i $ satisfies the inhomogeneous equation
$$  \qquad p_{i1}^2 - p_{i2}^2 -\, \cdots\, - p_{id}^2 \,\,=\,\,  m_i^2  \qquad
\hbox{for some constant $m_i$.} $$
Hence the determinant of its matrix $P_i$
equals ${\rm det}\,\,P_i\, =\, m_i^{2^k}$.
The vector $|\, i\, \rangle \,$ is now an eigenvector of $P_i$, by equation (\ref{eq:P^2}).
The algebraic relations among the spinor brackets, defined  below,
 would involve the masses $m_1,\ldots,m_n$ as parameters.
The study of such affine varieties for
 massive particles is left for future work.
\end{remark}

We write $\CC[p,z]$  for the polynomial ring
in $nd$ variables $p_{ij}$ and $n2^{k - 1}$ variables $z_{ij}$.
We view $I_{d,n}$ as an ideal in $\CC[p,z]$.
The quantities that represent interactions among
$n$ massless particles are certain elements in the
quotient ring
$$ R_{d,n} \,\, = \,\, \CC[p,z]/I_{d,n} .$$
We know from Theorem \ref{thm:parameters}
that $R_{d,n}$ is an integral domain for $\max(d,n) \geq 4$.

The following elements of $R_{d,n}$  are invariant under the  action
of the Lorentz group ${\rm SO}(1,d-1)$.
We define the {\em spinor brackets} of order two and three to be

\begin{equation} \label{eq:23brackets}
    \langle \, i j\, \rangle \,=\, \langle \, i \, | \, C \, | \, j\, \rangle
    \quad {\rm and} \quad \langle \, ij \,k\, \rangle \,=\, \langle \, i \, | \, C P_j \, | \, k\, \rangle.
\end{equation}
Here $i,j,k \in \{1,2,\ldots,n\}$.
Similarly, we define the {\em $\ell$-th order spinor brackets}:
\begin{equation}\label{eq:lthorderbrackets}
    \langle \, i_1 i_2 \,\cdots\, i_\ell\, \rangle 
    \,\,=\,\,  \langle \, i_1 \, | \, C P_{i_2} \,\cdots \,P_{i_{\ell - 1}} \, | \, i_\ell\, \rangle.
\end{equation}
Here $C$ is the charge conjugation matrix from
Example \ref{ex:Cmatrix} and Proposition \ref{prop:Cproperties}.

\begin{example}[$d=3$] \label{ex:d3}
Explicitly, the spinor brackets of order two and three are
$$ \begin{matrix} \langle i j \rangle & = &  \begin{small}
\begin{bmatrix} z_{i1} \!\! &\! 0 \end{bmatrix}
\begin{bmatrix}
                       p_{i3}  &       p_{i1} + p_{i2} \\
                       -p_{i1} + p_{i2} &       -p_{i3} \end{bmatrix}
\begin{bmatrix} \,\,0 & 1 \\ -1 & 0 \end{bmatrix}
\begin{bmatrix}   p_{j3} &  -p_{j1} + p_{j2} \\
                            p_{j1} + p_{j2} &        -p_{j3} \end{bmatrix}
\begin{bmatrix} z_{j1} \\ 0 \end{bmatrix} \end{small} \smallskip
\\
 & = &
-p_{i1} p_{j3} z_{i1} z_{j1} - p_{i2} p_{j3} z_{i1} z_{j1}
 + p_{i3} p_{j1} z_{i1} z_{j1} + p_{i3} p_{j2} z_{i1} z_{j1}, 
 \medskip \\
\langle ijk \rangle & = &
p_{i1} p_{j1} p_{k1} z_{i1} z_{k1} + p_{i1} p_{j1} p_{k2} z_{i1} z_{k1}
- p_{i1} p_{j2} p_{k1} z_{i1} z_{k1} - p_{i1} p_{j2} p_{k2} z_{i1} z_{k1} \\
& & - p_{i1} p_{j3} p_{k3} z_{i1} z_{k1}
+ p_{i2} p_{j1} p_{k1} z_{i1} z_{k1} + p_{i2} p_{j1} p_{k2} z_{i1} z_{k1} 
- p_{i2} p_{j2} p_{k1} z_{i1} z_{k1} \\ & & 
- p_{i2} p_{j2} p_{k2} z_{i1} z_{k1} 
 - p_{i2} p_{j3} p_{k3} z_{i1} z_{k1} + p_{i3} p_{j1} p_{k3} z_{i1} z_{k1}
+ p_{i3} p_{j2} p_{k3} z_{i1} z_{k1} \\ & &  - p_{i3} p_{j3} p_{k1} z_{i1} z_{k1} - p_{i3} p_{j3} p_{k2} z_{i1} z_{k1}.
\end{matrix}
$$
These are elements of the quotient ring $R_{3,n}$, so the constraints
(\ref{eq:massless}) and (\ref{eq:momentumconservation}) hold.          \end{example}

\begin{remark} \label{rmk:bracketsinvariant}
The spinor brackets $\langle i j \rangle$
and $\langle i j k \rangle$ are invariant under the 
 action of ${\rm SO}(1,d-1)$ by
left multiplication on the $z$-vectors and
conjugation on the $P$-matrices. For instance,
given $\Sigma \in \,\mathfrak{so}(1,d-1)$, the matrix
$g = {\rm exp}(\Sigma)$ transforms
    \begin{equation}
    \label{ij bracket}
        \langle ij\rangle \,\,  = \,\, z_i^TP_i^TCP_jz_j
    \end{equation}
    as desired:
$\,        (gz_i)^T(gP_ig^{-1})^T C(gP_jg^{-1})gz_j
        =z_i^TP_i^Tg^TCgP_jz_j = \langle ij \rangle $.
Here $\,g^T C g = C\,$ holds because $\,C\,\Sigma=-\Sigma^TC\,$ implies
 $   \, C\,{\rm exp}(\Sigma)\,=\,{\rm exp}(-\Sigma^T)\,C\,=\,{\rm exp}(\Sigma^T)^{-1}C$.
\end{remark}

\begin{remark}
In what follows we consider only spinor brackets with  $\ell = 2,3$.
Their varieties are already quite intriguing, as we shall see in
Sections \ref{sec5} and \ref{sec6}.
We do not claim that the brackets for $\ell \geq 4$ can be reduced
to $\ell=2,3$. It would be worthwhile to study such reductions from the perspective
of invariant theory.
\end{remark}

We next examine symmetries of the spinor brackets. 
To this end, we define 
\begin{equation} \label{eq:spinormatrix}
    S \, \, = \, \, \bigl(\,\langle \, ij \, \rangle \,\bigr)_{1 \leq i,j \leq n} \, \, = \,\, \bigl(
        | \, 1 \, \rangle , | \, 2 \, \rangle ,   \ldots,       | \, n \,  \rangle
    \bigr)^T \cdot\, C \cdot \bigl(
      | \, 1 \, \rangle , | \, 2 \, \rangle ,   \ldots,       | \, n \,  \rangle \bigr).
\end{equation}
This is an $n \times n$ matrix, obtained by multiplying matrices
of formats $n \times 2^k$, $\,2^k \times 2^k$ and $2^k \times n$.
Similarly, for each index $j \in \{1,2,\ldots,n\}$ we define
the $n \times n$ matrix
\begin{equation} \label{eq:spinormatrix2}
    T_j \, \, = \, \, \bigl(\,\langle \, ij\,k \, \rangle\, \bigr)_{1 \leq i,k\leq n}\, \, = \,\, \bigl(
        | \, 1 \, \rangle , | \, 2 \, \rangle ,   \ldots,       | \, n \,  \rangle
    \bigr)^T \cdot \,C \cdot P_j  \cdot \bigl(
      | \, 1 \, \rangle , | \, 2 \, \rangle ,   \ldots,       | \, n \,  \rangle \bigr).
\end{equation}

\begin{theorem}\label{thm:basicspinprops}
    The $n \times n$ matrix $S$ has rank $\le 2^{k}$ with zeros on the diagonal.
     If $k \equiv 0,3 \!\!\mod{4}$ then $S$ is symmetric;
       otherwise $S$ is skew symmetric. In symbols, 
    $$ \qquad
    \langle \, i\,i\,\rangle\, =\, 0 \quad {\rm and} \quad
        \langle \,ij\, \rangle \,=\, \pm\, \langle \,j\,i \, \rangle
        \qquad \hbox{for}\,\, \,\, 1 \leq i < j \leq n.
$$
The $n \times n$ matrix
    $T_j$ has rank $\le 2^{k  - 1}$ with zeros in the $j$th row
    and column.
        If $d \equiv 1,2,3,4 \!\! \mod 8$ then $T_j$ is symmetric;
        otherwise $T_j$ is skew symmetric. Thus,
            $$ \qquad
    \langle\, jj\,k\, \rangle \,=\, \langle\, ijj\, \rangle \,=\, 0 \quad {\rm and} \quad
    \langle \, ij\,k \, \rangle \,=\, \pm \,\langle \, kj\,i \, \rangle
 \qquad \hbox{for} \,\,\,\, 1 \leq i,j,k \leq n.    
    $$
    The sum of the $n$ matrices $T_j$ is the zero matrix, i.e.~$T_1 + T_2 + \cdots + T_n = 0$.
    \end{theorem}

\begin{proof}
    Since $C$ is of rank $2^{k}$, equation (\ref{eq:spinormatrix})
    implies that $S$ has rank at most $2^k$.
    In light of $P_i^2 = 0$, the vector $|\, i \, \rangle$ lies in the kernel of $P_i$,
     so $\langle \, i\, i \, \rangle = 
     \pm z_i^TCP_i | \, i \,\rangle = 0$.
    The (skew) symmetry property of $S$ follows from
    part (1) in Proposition \ref{prop:Cproperties}.

We now turn to $T_j$.
    By Corollary \ref{cor:rankP},
    the $2^k \times 2^k$ matrix
     $P_j$ has rank $2^{k - 1}$.
      It follows from (\ref{eq:spinormatrix2}) that $T_j$ has rank at most $2^{k - 1}$.
    Since $|\, j \, \rangle$ lies in the kernel of $P_j$, we conclude
    $\langle \, ijj \, \rangle = \langle \, i \, |\,C P_j\,| \, j\, \rangle = 0$,
    and similar for $\langle \, jjk \, \rangle$. The (skew) symmetry property of $T_j$ follows from Proposition \ref{prop:Cproperties} and equation~(\ref{eq:Cproperties}).
Finally,   momentum conservation $\sum_{j=1}^n P_j = 0$
implies the matrix identity $\sum_{j=1}^n T_j = 0$.
\end{proof}

We regard the tuple $T = (T_1,\ldots,T_n)$
as a tensor of format $n \times n \times n$.
By augmenting this with the matrix $S$,
we obtain a tensor $ST$ of format $n \times (n + 1) \times n$.

\begin{remark}
For the $\ell$-th order spinor brackets we can introduce the
   tensor $T^{(\ell)} = (\langle \, i_1 \cdots i_\ell \, \rangle)_{i_1, \dots, i_\ell}$ of size $n^\ell$. 
   In general, the following symmetries will hold: 
$$
\langle \, i_1 i_1 i_3 \,\cdots\, i_{\ell}\, \rangle \,=\, \langle\, i_1 \,\cdots\, i_j i_j\, \cdots\, i_{\ell}\, \rangle
\, =\, \langle\, i_1 \cdots i_{\ell - 2} i_{\ell} i_{\ell}\, \rangle \,=\, 0
\quad {\rm and} \quad \langle \, i_1\, \cdots \,i_\ell \, \rangle \,=\, \pm\langle \, i_\ell\, \cdots \,i_1 \, \rangle.
$$
In the latter equation,  the sign is positive when $\ell$ is even
and $d \equiv 0,1,6,7 \!\!\! \mod{8}$, or when  $\ell$ is odd
and $d \equiv 1,2,3,4 \!\!\! \mod{8}$. Otherwise we get a negative sign.
\end{remark}

The next two sections are concerned with
polynomial relations satisfied by the spinor brackets.
These relations define the kinematic varieties for massless particles
which are promised in our title.
We offer a preview for the case in Example~\ref{ex:d3}.
Our computations were performed with
the software  {\tt Macaulay2}~\cite{M2}.

\begin{example}[$d=3,k=1,n=4$] \label{ex:flatlanders}
We consider four particles in spacetime for flatlanders.
The six spinor brackets of order two form a skew symmetric matrix:
$$
S \,\,=\,\,  \begin{small} \begin{bmatrix}
0  & \langle 12 \rangle & \langle 13 \rangle & \langle 14 \rangle \\
 -\langle 12 \rangle & 0  & \langle 23 \rangle & \langle 24 \rangle \\
 -\langle 13 \rangle & -\langle 23 \rangle & 0  & \langle 34 \rangle \\
 -\langle 14 \rangle & -\langle 24 \rangle & -\langle 34 \rangle & 0  \\
\end{bmatrix}. \end{small}
$$
The $24$ spinor brackets of order three are the entries of 
four symmetric matrices:
$$
T_1 \,\,=\,\, \begin{small}
\begin{bmatrix}
0  & 0  & 0  & 0  \\
 0  & \langle 212 \rangle & \langle 213 \rangle & \langle 214 \rangle \\
 0  & \langle 213 \rangle & \langle 313 \rangle & \langle 314 \rangle \\
 0  & \langle 214 \rangle & \langle 314 \rangle & \langle 414 \rangle \\
\end{bmatrix} \end{small} \, , \qquad
T_2 \,\,=\,\,  \begin{small}
\begin{bmatrix}
 \langle 121 \rangle & 0  & \langle 123 \rangle & \langle 124 \rangle \\
 0  & 0  & 0  & 0  \\
 \langle 123 \rangle & 0  & \langle 323 \rangle & \langle 324 \rangle \\
 \langle 124 \rangle & 0  & \langle 324 \rangle & \langle 424 \rangle \\
\end{bmatrix}, \end{small}
$$
$$ T_3 \,\,=\,\,  \begin{small}
\begin{bmatrix} 
\langle 131 \rangle & \langle 132 \rangle & 0  & \langle 134 \rangle \\
 \langle 132 \rangle & \langle 232 \rangle & 0  & \langle 234 \rangle \\
 0  & 0  & 0  & 0  \\
 \langle 134 \rangle & \langle 234 \rangle & 0  & \langle 434 \rangle \\
\end{bmatrix} \end{small} \, , \qquad
T_4 \,\,=\,\, 
\begin{small}
\begin{bmatrix}
 \langle 141 \rangle & \langle 142 \rangle & \langle 143 \rangle & 0  \\
 \langle 142 \rangle & \langle 242 \rangle & \langle 243 \rangle & 0  \\
 \langle 143 \rangle & \langle 243 \rangle & \langle 343 \rangle & 0  \\
 0  & 0  & 0  & 0  \\
\end{bmatrix}. \end{small}
$$
We are interested in the subvariety of $\PP^5 \times \PP^{23}$
parametrized by the $30$ brackets.
This  {\em kinematic variety} is irreducible of dimension $4$ and its {\em multidegree} equals
\begin{equation}
\label{eq:multidegree34}
5 s^5 t^{19} \,+ \, 28 s^4 t^{20} \,+\,
24 s^3 t^{21}\,+\, 10 s^2 t^{22}\,+\, 2 s t^{23}
\,\,\in\,\,
H^*(\PP^5 \times \PP^{23}, \,\ZZ).
\end{equation}
The prime ideal of our variety is minimally generated
by the $10$ linear forms in
\begin{equation}
\label{eq:relations34a}
 T_1 \,+\, T_2 \,+\, T_3 \,+\, T_4 \,\,=\,\, 0,
 \end{equation}
together with $54 = 1 + 24 + 29$ quadrics.
First, there is the Pl\"ucker quadric
\begin{equation}
\label{eq:relations34b}
\langle 12 \rangle \langle 34 \rangle \,-
\langle 13 \rangle \langle 24 \rangle \,+
\langle 14 \rangle \langle 23\rangle \quad = \quad
{\rm Pfaffian}(S), 
\end{equation}
which ensures that $S$ has rank two.
Next, we have $24$ binomial quadrics like
\begin{equation}
\label{eq:relations34c}
 \langle ij\,k \rangle \langle l j \,m \rangle \, - \,
\langle ij\,m \rangle \langle l j \,k \rangle.
\end{equation}
For each $j$, we have six such binomials. They are the $2 \times 2$ minors
 which ensure that the matrix $T_j$ has rank $\leq 1$.
Finally, we have $29$ bilinear relations, such as
\begin{equation}
\label{eq:flatrel} \langle 12 \rangle \langle 324 \rangle - \langle 34 \rangle \langle 142 \rangle
\quad {\rm and} \quad
  \langle 12 \rangle  \langle 243 \rangle
 - \langle 13 \rangle \langle 242 \rangle 
 + \langle 23 \rangle \langle 142 \rangle. 
 \end{equation}
These ensure that the 
$4 \times 20$ matrix $(S,T_1,T_2,T_3,T_4)$ has rank $\leq 2$.
We obtain (\ref{eq:flatrel}) from the
$3 \times 3$ minors of this flattening
of the $4 \times 5 \times 4$ tensor $ST$.
Our prime ideal is generated 
by (\ref{eq:relations34a}),
 (\ref{eq:relations34b}),
 (\ref{eq:relations34c}) and quadrics like (\ref{eq:flatrel}).
 See also Conjecture~\ref{conj:d=3}.
 \end{example}

\section{Varieties in Matrix Spaces}
\label{sec5}

The order two spinor brackets $\langle i j \rangle$ are the $\binom{n}{2}$  entries
of an $n \times n$ matrix $S$ which is either symmetric or skew symmetric.
The {\em kinematic variety} $\mathcal{K}^{(2)}_{d,n}$
consists of such matrices $S$.
This is an irreducible variety in  $\PP^{\binom{n}{2}-1}$.
We seek its prime~ideal.

\begin{theorem}
\label{thm:ordertwosmall}
For $d=3$, the ideal of $\mathcal{K}^{(2)}_{3,n}$ 
is generated by the $4 \times 4$ Pfaffians of
a skew symmetric $n \times n$ matrix, so this kinematic
variety is the Grassmannian ${\rm Gr}(2,n)$.
For $d=4,5$, the ideal is generated by the $6 \times 6$ Pfaffians
of a skew symmetric $n \times n$ matrix, and hence
   $\mathcal{K}^{(2)}_{4,n}=\mathcal{K}^{(2)}_{5,n}$   is the secant variety of ${\rm Gr}(2,n)$.
\end{theorem}

\begin{proof}
For $d = 3,4,5$ we have $k \equiv1,2 \!\! \mod{4}$.
 Theorem \ref{thm:basicspinprops} tells us that
$S$ is skew symmetric of rank $\leq 2^k$.
Hence $\mathcal{K}^{(2)}_{d,n}$ is contained in
${\rm Gr}(2,n)$ and its secant variety respectively.
To see that they are equal, we checked (using {\tt Macaulay2}) that 
the entries of the column vector $\,|\, i \rangle\,$ are algebraically independent.
Thus~(\ref{eq:spinormatrix})~is a generic
skew symmetric $n \times n$ matrix of rank $\leq 2^k$.
This might fail for $d \geq 6$.
\end{proof}

\begin{remark}
The dimensions of Grassmannians and their secant varieties
are well known; see e.g.~\cite{CGG}. For the kinematic varieties
in Theorem \ref{thm:ordertwosmall}, we have
\begin{equation}\label{eq:grassdim} {\rm dim}\, \mathcal{K}^{(2)}_{3,n} \,=\, 2n-4
\quad {\rm and} \quad 
{\rm dim} \, \mathcal{K}^{(2)}_{4,n} \,=\,{\rm dim} \, \mathcal{K}^{(2)}_{5,n} \,=\,
 4n-11 \qquad {\rm for} \,\,\, n \geq 4.
\end{equation}
 \end{remark}

Next we consider spacetime dimensions $d = 6,7,8,9$.
Here $k = 3,4$, so $S$ is a symmetric $n \times n$ matrix.
It still has zeros on the diagonal and ${\rm rank} \,S \leq 2^k$,
by Theorem \ref{thm:basicspinprops}.
The variety of such matrices is irreducible, and its prime ideal
is generated by the $(2^k+1) \times (2^k+1)$ minors.
This is proved for $4 \times 4$ minors in \cite[Theorem 3.5]{DFRS},
but the proof is the same for matrix of fixed size larger than $4$.

\begin{conjecture} \label{conj:K2}
For $d \equiv 0,1,6,7 \!\!\mod{8}$,
the kinematic variety $\mathcal{K}^{(2)}_{d,n}$
consists of all symmetric $n \times n$ matrices
with zero diagonal and rank $\leq 2^{\lfloor d/2 \rfloor}$. 
For $d \equiv 2,3,4,5 \!\!\mod{8}$,
$\,\mathcal{K}^{(2)}_{d,n}$ is the variety of
 skew symmetric $n \times n$ matrices
of rank $\leq 2^{\lfloor d/2 \rfloor}$.
\end{conjecture}

The determinantal varieties in Conjecture \ref{conj:K2}
are irreducible and their dimensions are known. For instance,
going one step beyond (\ref{eq:grassdim}), we would have
\begin{equation}
\label{eq:dim29}
 {\rm dim}\,\mathcal{K}^{(2)}_{d,n} \,\,= \,\,
7n-29 \quad {\rm for} \,\, d=6,7\,\,{\rm and}\,\,n \geq 7 . 
\end{equation}
To prove Conjecture \ref{conj:K2}, it suffices to show that
the dimension $\mathcal{K}^{(2)}_{d,n}$
equals that of the determinantal variety.
Both varieties are irreducible, and the former is contained in the latter.
For instance, we verified (\ref{eq:dim29}) numerically for $n=7,8,9$.

The difficulty in extending the proof of Theorem 
\ref{thm:ordertwosmall} to higher dimensions $d \geq 6$ is that the
entries of the vector $|\, i \,\rangle$ are 
no longer algebraically independent.

\begin{remark}
For $d \ge 6$, the spinors $|\,i \,\rangle$ live in a proper subvariety of $\PP^{2^k-1}$. For instance, for $d=6$,
 the eight coordinates of $| \,i\, \rangle$ satisfy the algebraic relation
\begin{equation}
    |\, i \,\rangle_\varnothing \cdot |\, i \,\rangle_{123} \,\,-\,\, |\, i \,\rangle_1 \cdot 
    |\, i \,\rangle_{23} \,\,+\,\, |\, i \,\rangle_2 \cdot |\, i \,\rangle_ {13}
    \,\,-\,\, |\, i \,\rangle_3 \cdot |\, i \,\rangle_{12} \,\,\,\,=\,\,\,\, 0.
\end{equation}
For $k \geq 4$, we used {\tt HomotopyContinuation.jl} \cite{julia} 
to compute the codimension of the variety 
parametrized by $| i \rangle$. Based on this, we conjecture the formulas
\begin{equation}
\label{eq:eulerian}
{\rm codim}(2k)\, =\, 2^{k - 1} - 2(k - 1) \quad {\rm and} \quad {\rm codim}(2k + 1) \,=\, {\rm codim}(2k) - 1.
\end{equation}
\end{remark}

We now relate $\mathcal{K}^{(2)}_{d,n}$
 to  the {\em spinor-helicity formalism} in \cite[Section~2.2]{EH}.
For this, we assume that $d=2k$ is even. As seen
in Example \ref{ex:Pmatrix}, the momentum space Dirac matrix $P$ is anti-block diagonal
with two blocks of size $2^{k-1}\times 2^{k-1}$:
\begin{align}
    P \,\,=\,\, \begin{bmatrix}
        0 & P' \\
        P'' & 0
    \end{bmatrix}.
\end{align}
This reflects the fact that the action of $\mathrm{SO}(1,d-1)$ decomposes into two irreducible representations.
  We consider 
the blocks $P'$ and $P''$ separately, and define both angle and square spinor brackets. Let $\tilde{x}_i$ be the vector consisting of the first $2^{k-1}$ entries of $z_i$ and ${x_i}$ the vector consisting of the last $2^{k-1}$ entries of $z_i$. We set
$$
 |{i}\rangle\,=\,P'_i{x}_i \,, \quad
  \langle{i}|\,=\,|{i}\rangle^T\, \qquad {\rm and} \qquad
  |i]\,=\,P''_i \tilde{x}_i \,,\quad [i| \,=\,|i]^T.
  $$
  By Proposition \ref{prop:Cproperties},
the $C$ matrix is either block diagonal or anti-block diagonal:
$$
    C \,=\, \begin{bmatrix}
        C' & 0 \\
        0 & C''
    \end{bmatrix} \,\,\,\,\hbox{if $k$ is even} \quad {\rm and} \quad
    C\, =\, \begin{bmatrix}
        0 & C'' \\
        C' & 0
    \end{bmatrix} \,\,\,\,\hbox{if $k$ is odd.}
    $$
In analogy to (\ref{eq:lthorderbrackets}),
we define two types of elements in the ring $R_{d,n}$ as follows:
\begin{equation}
\label{eq:twobrackets1}
    \langle i_1i_2\cdots i_\ell\rangle\,=\, \langle i_1|\,C'P'_{i_2}\cdots P'_{i_{\ell-1}}|\,i_\ell\rangle
    \quad {\rm and} \quad [i_1i_2\cdots i_\ell]\,=\, [i_1|\,C''P''_{i_2}\cdots P''_{i_{\ell-1}}| \, i_\ell].
\end{equation}
Moreover, it  makes sense to mix and match the brackets, so we can also define
\begin{equation}
\label{eq:twobrackets2}
    \langle i_1i_2\cdots i_\ell]\,=\, \langle i_1|\,C''P''_{i_2}\cdots P''_{i_{\ell-1}}|\,i_\ell]
    \quad {\rm and} \quad [i_1i_2\cdots i_\ell\rangle \,=\, [i_1|\,C'P'_{i_2}\cdots P'_{i_{\ell-1}}|\,i_\ell\rangle.
\end{equation}
If $k$ is even, then the brackets in (\ref{eq:twobrackets1}) are Lorentz invariant,
in the sense of Remark \ref{rmk:bracketsinvariant}. Likewise,
if $k$ is odd, then the brackets in (\ref{eq:twobrackets2}) are Lorentz invariant.
 
We now focus on $d=4$ and $\ell=2$, and
we erase the $z$ parameters by~setting
 \begin{align}
x_i \,\,=\,\, \tilde{x}_i \,\, = \,\, (p_{i1}+p_{i2})^{-1/2} \quad
{\rm for} \,\,\,\,i=1,2,\ldots,n.
 \end{align} 
 The $\langle ij \rangle$ and $[ij]$
form skew symmetric $n \times n$ matrices
of rank $2$ whose product is the zero matrix.
This recovers the  {\em spinor helicity variety} in \cite[Example~1.1]{EPS}:
$$ \mathrm{SH}(2,n,0)\,\,\subset \,\,\mathrm{Gr}(2,n)\times \mathrm{Gr}(2,n) \,\,\subset\,\,
\mathbb{P}^{\binom{n}{2}-1}\times\mathbb{P}^{\binom{n}{2}-1}.$$
It would be interesting to 
study the varieties given by 
 (\ref{eq:twobrackets1}) and (\ref{eq:twobrackets2}) for 
 $d=6,8,\ldots$.
In the present paper we limit ourselves to the
angle brackets in~(\ref{eq:23brackets}).
\section{Varieties in Tensor Spaces}
\label{sec6}

Let $\mathcal{K}^{(3)}_{d,n}$ denote the kinematic variety of second  and third order spinor brackets.
This is an irreducible subvariety of
$\,\PP^{\binom{n}{2} - 1} \times \PP^{K -1}$,
where $K = n \cdot \binom{n}{2}$ when the slices $T_j$ are symmetric, i.e.\ $d \equiv 1,2,3,4 \!\! \mod{8}$, and 
$K = n \cdot\binom{n - 1}{2}$ when the slices $T_j$ are skew symmetric;
see Theorem \ref{thm:basicspinprops}.
The prime ideal of $\mathcal{K}^{(3)}_{d,n}$  consists of all
polynomial relations among the spinor brackets $\langle i j \rangle$ and $\langle i j \,k \rangle$. 
It is $\ZZ^2$-graded because 
$\langle i j \rangle$ and $\langle i j\, k \rangle$ are homogeneous
elements in the $\ZZ^2$-graded ring $R_{d,n} = \CC[p,z]/I_{d,n}$, of degrees 
$(2,2)$ and $(3,2)$ respectively; see Example  \ref{ex:d3}.

Points in the kinematic variety  $\mathcal{K}^{(3)}_{d,n}$ are
  $n \times (n + 1)\times n $ tensors $ST$. 
 The tensor slices $S,T_1,\ldots,T_n$ 
 are $n \times n$ matrices
 which are symmetric or skew symmetric, depending
 on the residue classes of $k = \lfloor d/2 \rfloor$ 
 modulo $4$ and $d$ modulo~$8$.
 We refer to \cite{JML, MS}
 for basics on tensors,  such as
 rank, slices and flattenings.
 
 \begin{remark}[Three particles]
 Let $n=3$ and $d \geq 4$. 
  The kinematic variety $\mathcal{K}^{(3)}_{d,3}$ is
 empty since $T_1=T_2=T_3 = 0$. To see this,
 we first note that  $\langle ijk \rangle =0$
 for $i,j,k$ distinct, since $T_1 + T_2 + T_3 = 0$. Remark~\ref{rmk:maxideal} implies 
 $P_iP_j + P_j P_i = 0_{2^k}$, and hence
$\,\langle ij \,i \rangle = z_i^T P_i^TCP_jP_iz_i = \pm z_i^TCP_jP_i^2z_i = 0$.
Therefore, $\langle ijk \rangle =0$ for all $i,j,k$, showing the variety is empty.
The matrix $S$ is unconstrained, and hence
$ \mathcal{K}^{(2)}_{d,3} = \PP^2$.
 \end{remark}
 
 From now on we assume $n \geq 4$.
We start with the case
$d=3$. Here, the matrix $S$ is skew symmetric and the slices  $T_j$ are symmetric.
 Our kinematic variety $\mathcal{K}^{(3)}_{3,n}$ lives
 in the tensor space $\PP^{\binom{n}{2}-1} \times \PP^{n \binom{n}{2}-1}$.
 See Example~\ref{ex:flatlanders} for the case $n=4$.
 
To study $\mathcal{K}^{(3)}_{3,n}$ for $n \geq 5$, we introduce some terminology for third order tensors. 
Recall from \cite[Chapter~2]{JML} that
 the \textit{multilinear rank} of a tensor is the triple of the ranks of its  flattenings.
  A \textit{Tucker decomposition} of an $n_1 \times n_2 \times n_3$ tensor $T$ with multilinear rank $(r_1, r_2, r_3)$ is a factorization of $T$ into an $r_1 \times r_2 \times r_3$ \textit{core tensor} $K$ and matrices $U_1, U_2, U_3$ of sizes $n_1 \times r_1$, $n_2 \times r_2$ and $n_3 \times r_3$ respectively.
 
\begin{conjecture} \label{conj:d=3}
    The variety  $\mathcal{K}^{(3)}_{3,n}$ has dimension
    $3n-8$. Its points are
      tensors~$ST$ of multilinear rank $ \leq (2,4,2)$.
      Here $S$ is skew symmetric of rank $2$, and
      the $T_j$ are symmetric of rank $\leq 1$,
      summing to $0$, with zeros in the $j$-th row and~column.
      The prime ideal of  $\mathcal{K}^{(3)}_{3,n}$ is generated by polynomials
of degree $\leq 3$, namely    the entries of  $T_1 +T_2 + \cdots + T_n $,
    $4\, \times 4$ Pfaffians of $S$,
     $\,2 \times 2$ minors of the $T_j$'s, 
      $\,3 \times 3$ minors of the flattening $(S, T_1, \dots, T_n)$,
     and quadrics that are mixed Pfaffians.
     \end{conjecture}

We now present evidence for this conjecture. First and foremost,
every $n \times (n+1) \times n$ tensor $ST$ in the kinematic variety $ \mathcal{K}^{(3)}_{3,n}$
 has a Tucker decomposition whose core tensor $K$
has format  $2 \times 4 \times 2$. The four $2 \times 2$ slices of the core $K$ are
    $$
    K_1 \,=\, \begin{bmatrix}
        0 & 1\\
        -1 & 0 
    \end{bmatrix}, \quad K_2 \,=\, \begin{bmatrix}
        0 & 1\\
        1 & 0 
    \end{bmatrix}, \quad K_3 \,=\, \begin{bmatrix}
        1 & 0\\
        0 & 0 
    \end{bmatrix},\quad K_4 \,=\, \begin{bmatrix}
        0 & 0\\
        0 & 1
    \end{bmatrix}.
    $$
    The corresponding factor matrices of format $2 \times n$ and $4 \times (n+1)$ are
    $$
    U_1 \,=\, U_3 \,= \,\bigl(
      | \, 1 \, \rangle , | \, 2 \, \rangle ,   \ldots,       | \, n \,  \rangle \bigr), \,\,\,\, 
      U_2 \,=\, \begin{bmatrix}
      1 & 0 \\ 0 & \overline{U}_2 
      \end{bmatrix} \,=\,
      \begin{small}
      \begin{bmatrix}
          1 & 0 & \cdots & 0\\
          0 & -p_{13} &  \cdots &  -p_{n3}\\
          0 & p_{11} + p_{12} &  \cdots &  p_{n1} + p_{n2}\\
          0 & p_{11} - p_{12} &  \cdots &  p_{n1} - p_{n2}\\
      \end{bmatrix}. \end{small}
    $$
    The matrices whose columns are the spinors (\ref{eq:paramofcolumnspace}) can be written as follows:
$$            U_1 \,=\, U_3 \,   \,=\,\, \begin{bmatrix}
        -1 & 0 & 0\\
        0 & 1 & 0
    \end{bmatrix} \cdot \overline{U}_2 \cdot 
    \begin{small} \begin{bmatrix}
        z_{11} & 0 & \cdots & 0\\
        0 & z_{21} & \cdots & 0\\
        \vdots & \vdots & \ddots & \vdots\\
        0 & 0 & \cdots & z_{n1}\\
    \end{bmatrix}. \end{small}
    $$
    
    From the Tucker decomposition we see that the $n \times n$ slices of $ST$ are 
    \begin{equation}
\label{eq:tuckerflat}
\!        S \,=\, U_1^T K_1 U_1\, \,\,\, {\rm and}\,\,\,\,
     T_j \,=\, U_1^T\bigl(-p_{j3}K_2 + (p_{j1} + p_{j2})K_3 + (p_{j1} - p_{j2})K_4\bigr)\, U_1.
\end{equation}
    The $2 \times n$ matrix $U_1$ is generic, as in the
    proof of Theorem \ref{thm:ordertwosmall}. This implies
        that $S$ is a generic skew symmetric $n \times n$ matrix of rank $2$.
    The space of symmetric $2 \times 2$ matrices has the basis $\{K_2,K_3,K_4\}$.
    The parenthesized matrix in (\ref{eq:tuckerflat}) has
    determinant $p_{j1}^2 - p_{j2}^2 - p_{j3}^2$,
    and hence $T_j$ is a symmetric 
    $n \times n$ matrix of rank $\leq 1$. 
    
    More precisely, $T_j$ is the outer product of the
    following vector with itself:
    $$
    v_j \,\,=\,\, (p_{j1} + p_{j2})^{-\frac{1}{2}} \begin{bmatrix}
        p_{j1} + p_{j2} & -p_{j1} + p_{j2}
    \end{bmatrix}U_1\,\, =\,\, (p_{j1} + p_{j2})^{-\frac{1}{2}}  u_j \,U_1.
    $$
        Since $u_j$ is in the kernel of 
        the $2 \times 2 $ matrix $P_j$, the $j$-th coordinate of $v_j$ is zero.
        Therefore the $j$-th row and column of $T_j$ are zero.
    In summary, we have shown that the 
    $n \times (n+1) \times n$ tensor $ST$ has all the properties 
    stated in Conjecture \ref{conj:d=3}.

    We now return to arbitrary spacetime dimension $d \geq 3$.
Our ultimate goal is to understand 
 $\mathcal{K}^{(3)}_{d,n}$ for arbitrary $d$ and $n$.
 A first step is to determine the dimension of this kinematic variety.
For  small cases, we computed this
using the numerical software
{\tt HomotopyContinuation.jl} \cite{julia}.
This computation is nontrivial because the
momenta $p_i$ must satisfy the constraints studied in
Section~\ref{sec2}.

\begin{proposition} \label{prop:dimensions}
The dimensions of the kinematic varieties $\mathcal{K}^{(3)}_{d,n}$ are as follows:

$$ \begin{small}
\begin{matrix}
    d \backslash n  & 4 & 5 & 6 & 7 & 8 & 9 & 10 & 11 & 12\\
    4 &  8 & 13 & 18 & 23 & 28 & 33 & 38 & 43 & 48\\
    5 &  7 & 13 & 19 & 25 & 31 & 37 & 43 & 49 & 55\\
    6 &  9 & 20 & 30 & 40 & 49 & 58 & 67 & 76 & 85\\
    7 &  9 & 20 & 30 & 40 & 50 & 60 & 70 & 80 & 90\\
    8 &  10 & 28 & 51 & 67 & 82 & 97 & 112 & 127 & 142\\
    9 &  15 & 33 & 49 & 65 & 81 & 97 & 113 & 129 & 145 \\
\end{matrix} \end{small}
$$
\end{proposition}

 Proposition \ref{prop:dimensions} suggests
that $\,{\rm dim}\,\mathcal{K}^{(3)}_{d,n}\,$ is a linear
function in $n$ when $d$ is fixed. We note that the slope is
$5,6,9,10,15,16$ for $d=4,5,6,7,8,9$.
Following \cite{PRRVZ}, we are 
especially interested in the case $d=5$.
Proposition \ref{prop:dimensions} suggests that 
$\mathcal{K}^{(3)}_{5,n}$ has dimension $6n-17$.
We next characterize this variety for five particles.

\begin{example}
Fix $d = 5, k = 2, n = 5$.
The variety $\mathcal{K}^{(3)}_{5,5} \subset \PP^{9} \times \PP^{29}$ is irreducible of dimension $13$ and degree $1761$. Its class in
$H^*(\PP^9 \times \PP^{29},\ZZ)$ is the multidegree
$$
80 s^8 t^{17} \! +   265 s^7 t^{18}  \! +   430 s^6 t^{19}  \!+  450 s^5 t^{20}  \!+   320 s^4 t^{21}
\! +   155 s^3 t^{22}  \!+ 
50 s^2 t^{23} \!+  10 s t^{24}\! +  t^{25}
.$$
Modulo the $10$ linear forms in
$\,T_1 + T_2 + T_3 + T_4 + T_5 $, the ideal is generated by
$25$ quadrics, $15$ cubics and $5$ quartics. 
Each $T_j$ is a skew symmetric $5 \times 5$ matrix with one
zero row, so it contributes one Pfaffian
$
\langle ijk \rangle \langle \ell jm \rangle - \langle ij\ell \rangle \langle k jm \rangle + \langle ijm \rangle \langle k j\ell \rangle
$.
The other $20$ quadrics are bilinear Pfaffians
$
\langle ij \rangle \langle ki\ell \rangle - \langle ik \rangle \langle ji\ell \rangle + \langle i\ell \rangle \langle jik \rangle$, e.g.
\begin{equation}
\label{eq:specialslices1}
 \hbox{$4 \times 4$ Pfaffians of} \quad
\begin{footnotesize}
\begin{bmatrix}
0  & \langle 12 \rangle & \langle 13 \rangle   & \langle 14 \rangle  & \langle 15 \rangle   \\
 -\langle 12 \rangle & 0 & \langle 213 \rangle   & \langle 214 \rangle  & \langle 215 \rangle   \\
 -\langle 13 \rangle & -\langle 213 \rangle  & 0  & \langle 314 \rangle  & \langle 315 \rangle   \\
 -\langle 14 \rangle & -\langle 214 \rangle  & -\langle 314 \rangle & 0  & \langle 415 \rangle   \\
  -\langle 15 \rangle & -\langle 215 \rangle  & -\langle 315 \rangle & \langle -415 \rangle & 0   \\
\end{bmatrix}. \end{footnotesize} 
\end{equation}
Next, we have $15$ cubics which ensure that the  flattening $(S, T_1, T_2, T_3, T_4, T_5)$
has rank $\leq 4$. One of them is
              $ \langle 213 \rangle\langle 123 \rangle  \langle 435 \rangle
              -\langle 213 \rangle \langle 325 \rangle\langle 134 \rangle  +\langle 213 \rangle\langle 324 \rangle \langle 135 \rangle $ $+\,\langle 314 \rangle\langle 123 \rangle \langle 235 \rangle
            -\langle 314 \rangle \langle 325 \rangle\langle 132 \rangle 
              -\langle 315 \rangle\langle 123 \rangle \langle 234 \rangle 
+ \langle 315 \rangle\langle 324 \rangle\langle 132 \rangle$.
Finally, the $5$ quartic generators come from the $4 \times 4$ minors of the following slices:
\begin{equation}
\label{eq:specialslices2}
\begin{footnotesize} \!\!\!\!\!\!
\begin{bmatrix}
0  & \langle 12 \rangle & \langle 13 \rangle   & \langle 14 \rangle  & \langle 15 \rangle   \\
0  & 0 & 0 & 0 & 0 \\
0  & 0 & \langle 123 \rangle   & \langle 124 \rangle  & \langle 125 \rangle   \\
0  &  \langle 132 \rangle & 0  & \langle 134 \rangle  & \langle 135 \rangle   \\
0  &  \langle 142 \rangle & \langle 143 \rangle & 0   & \langle 145 \rangle   \\
0 &  \langle 152 \rangle & \langle 153 \rangle & \langle 154 \rangle & 0 \\
\end{bmatrix}  \, , \,\,\,
\begin{bmatrix}
-\langle 12 \rangle & 0 & \langle 23 \rangle   & \langle 24 \rangle  & \! \langle 25 \rangle   \\
0 & 0 & \langle 213 \rangle   & \langle 214 \rangle  & \! \langle 215 \rangle   \\
0  & 0 & 0 & 0 & 0 \\
-\langle 132 \rangle & 0  & 0 & \langle 234 \rangle  &\! \langle 235 \rangle   \\
-\langle 142 \rangle & 0 & \langle 243 \rangle & 0   &\! \langle 245 \rangle   \\ -\langle 152 \rangle & 0 & \langle 253 \rangle & \langle 254 \rangle & 0 \\
\end{bmatrix}, \,\end{footnotesize}\, {\rm etc} \ldots
\end{equation}
\end{example}

We conclude this paper by summarizing what we know in general
about the kinematic variety $\mathcal{K}^{(3)}_{5,n}$ 
for $n \geq 4$ massless particles in spacetime dimension $d=5$.
The points in    $\mathcal{K}^{(3)}_{5,n}$ 
   are $n  \times (n + 1)\times n$ tensors $ST$. 
   The matrix $S$ is skew symmetric of rank $\le 4$ and the slices $T_j$ are skew symmetric of rank $\le 2$, summing to $0$, and the $j$-th row and column of $T_j$ is zero.
   We saw this in Theorem \ref{thm:basicspinprops}.
    From these constraints we see that the $6 \times 6$ Pfaffians of $S$ and the $4 \times 4$ Pfaffians of $T_j$ vanish. 
   
   Our next proposition explains the mixed relations we found
   in (\ref{eq:specialslices1}) and the cubics arising from the flattening $(S, T_1, T_2, T_3, T_4, T_5)$
   of the $5 \times 6 \times 5$ tensor $ST$.
   
   \begin{proposition}
For each index $j \in \{1,\ldots,n\}$, the skew symmetric $n \times n$ matrix
\begin{equation}
\label{eq:mixedmatrix}
\!\!    \bigl(\, |1\rangle, \ldots, | j\!-\!1 \rangle, z_j , | j\!+\!1 \rangle,  \ldots,       | \, n \,  \rangle
    \bigr)^T\! \! \cdot C \cdot P_j  \cdot\bigl( \,|1\rangle, \ldots,  | j\!-\!1 \rangle, z_j ,  | j\!+\!1 \rangle,   \ldots, | \, n \,  \rangle
    \bigr) 
\end{equation} 
has the spinor brackets $\langle ij \rangle$ and $\langle i j k \rangle$ for its
non-zero entries.
It has rank  $\leq 2$ on the kinematic variety $\mathcal{K}^{(3)}_{5,n}$,
so the $4 \times 4$ Pfaffians give bilinear ideal generators.
Furthermore, the $n \times (n^2 +n)$ matrix flattening
$(S,T_1,\dots,T_n)$ of the tensor $ST$ has rank $\leq 4$ on $\mathcal{K}^{(3)}_{5,n}$.
It contributes $6 \times 6$ Pfaffians to the ideal generators.
   \end{proposition}

\begin{proof}
The statement about (\ref{eq:mixedmatrix}) holds because the
$4 \times 4$ matrix $C \cdot P_j$ is skew symmetric and its rank is $2$.
See (\ref{eq:Cproperties}) and Example \ref{ex:Cmatrix}.
The bound on the rank of $(S,T_1,\dots,T_n)$
follows from  the matrix factorizations in
(\ref{eq:spinormatrix})  and (\ref{eq:spinormatrix2}).
\end{proof}   

We also consider the $n \times (n+1)$ slices of the tensor $ST$ 
obtained by fixing the the first or last index. Numerical
computations show that these matrices have rank $\leq 3$ on $\mathcal{K}^{(3)}_{5,n}$. 
We also discovered that the multilinear rank of $ST $ equals $(4,6,4)$.
Our final question concerns the tensor rank of $ST$
on the variety $\mathcal{K}^{(3)}_{5,n}$.
We can prove that 
 {\em the tensor rank of $ST$  is at least $5$}.
We show this by evaluating the Strassen invariant \cite[Example 9.21]{MS}
on subtensors of format $3 \times 3 \times 3$.
At present we do not have a conjecture for the tensor rank on
$\mathcal{K}^{(3)}_{5,n}$ for general $n$.

\smallskip 

We conclude  by expressing the hope that our tensor varieties for arbitrary $d$ 
will be of interest to physicists.  Potential points of connection in the
physics literature are the studies of spinor helicity for $d=6 $ in \cite{CC} and for
 $d=10$ in~\cite{CHC}.   The case $d=10$ is especially relevant for $ \mathcal{N}=4 $ 
 supersymmetric Yang-Mills theory. Perhaps some of the insights gained there can help
 in proving our conjectures?

\bigskip
\medskip

\noindent {\bf Acknowledgements.} 
We are grateful to Andrzej Pokraka for his comments and to
 Yassine El Maazouz for helping us
with the proof of Theorem~\ref{thm:parameters}.

\bigskip \bigskip

\begin{footnotesize}
\noindent {\bf Funding statement}:
Smita Rajan is supported by a Graduate Fellowship from the US National Science Foundation.
Bernd Sturmfels is supported by the European Union (ERC, UNIVERSE PLUS, 101118787).
Views and opinions expressed are however those of the authors only and do not
necessarily reflect those of the European Union or the European Research Council
Executive Agency. Neither the European Union nor the granting authority can
be held responsible for them.
\end{footnotesize}

\begin{small}

\end{small}

\end{document}